\documentclass[12pt]{article}

\usepackage{amssymb,graphicx,amsthm}
\usepackage{amsmath}
\usepackage{fancyvrb}
\usepackage{tikz}

\def\Y{{\cal Y}}
\def\I{{\cal I}}

\def\S{{\cal S}}
\def\L{{\cal L}}

\def\S{{\cal S}}
\newcommand{\LHD}{\mathbin{< \hbox to -.415em{}\hbox{\vrule
            height .51em depth .04em}\hbox to .2em{}}}

\counterwithin{figure}{section}

\newtheorem{thm}{Theorem}[section]

\newtheorem{prop}[thm]{Proposition}

\begin{document}

\title{{\bf On the minimum blocking semioval in $PG(2,11)$}}

\author{Jeremy M. Dover}

\maketitle

\begin{abstract}
A blocking semioval is a set of points in a projective plane that is both a blocking set (i.e., every line meets the set, but the set contains no line) and a semioval (i.e., there is a unique tangent line at each point). The smallest size of a blocking semioval is known for all finite projective planes of order less than 11; we investigate the situation in $PG(2,11)$.
\end{abstract}

\section{Introduction}
In a projective plane $\pi$, a {\bf semioval} is a set of points $S$ such that there is a unique tangent line (i.e., line with one point of contact) at each point. The size of a semioval $\S$ in $PG(2,q)$ is known to satisfy $q+1 \le |\S| \le q\sqrt{q}+1$. As defined by Kiss~\cite{kiss}, a semioval $\S$ is {\em small} if it satisfies $|\S| \le 3q+3$. Several articles, including Kiss, et al.~\cite{kmp} and Bartoli~\cite{bartoli}, investigate the spectrum of sizes for small semiovals in Desarguesian planes.

A set of points $S$ in $\pi$ is called a {\bf blocking set} if every line of $\pi$ meets $S$ in at least one point, but $S$ contains no line. A set of points in $\pi$ is called a {\bf blocking semioval} if it is both a blocking set and a semioval. The smallest possible size of a blocking semioval is known for all planes of order less than 11. In the smallest cases, the author~\cite{dover1} showed there are no blocking semiovals in $PG(2,2)$, and the smallest possible sizes in $PG(2,q)$ for $q=3,4,5,7$ are $6, 9, 11, {\rm and}\; 16$, respectively. In $PG(2,8)$ the author~\cite{dover2} exhibited a blocking semioval of size 19, and showed that no smaller blocking semiovals can exist.

Regarding the planes of order 9, the author~\cite{dover2} showed that no plane of order 9 can have a blocking semioval of size 20 or smaller. Nakagawa and Suetake~\cite{suetake3} had previously shown the existence of blocking semiovals of size 21 in the three non-Desarguesian planes of order 9, and Dover, Mellinger and Wantz~\cite{dmw} exhibited a blocking semioval of size 21 in $PG(2,9)$.

In this article, we continue this program by studying the smallest blocking semioval in $PG(2,11)$.

\section{Some introductory analysis}
To begin scoping the problem, we first appeal to the two best known lower bounds for the size of blocking semiovals. For any projective plane of order $q$, the author~\cite{dover2} shows that a blocking semioval must have at least $2q + \sqrt{2q-\frac{47}{4}}-\frac12$ points. Specializing to the $q=11$ case, we calculate that a blocking semioval in $PG(2,11)$ must have at least 25 points. Since $PG(2,11)$ is Desarguesian, the lower bound of H\'{e}ger and Tak\'{a}ts~\cite[Corollary 33]{hegertakats} also applies, which states that a blocking semioval in $PG(2,q)$ must have at least $\frac94q-3$ points. However, this result provides a bound less than 25, from which we conclude that the smallest possible size for a blocking semioval in $PG(2,11)$ is 25 points.

Let $\S$ be a putative blocking semioval of size 25 in $PG(2,11)$. The author~\cite{dover3} (Proposition 2.1 and Theorem 2.2) shows that $\S$ cannot fully contain a line, nor can it have an $11$-secant. We then apply Theorem 3.1 from Dover~\cite{dover1}, which states that if $\S$ has an (11-$k$)-secant for $1\leq k < 10$, then $\S$ has at least $11 \frac{3k+4}{k+2} - k$ points. Evaluating all of these bounds as $k$ varies from 1 to 10, we find that $\S$ cannot have any 7-secants, 8-secants or 9-secants.

One possibility is that $\S$ could have a 10-secant, in which case Theorem 4.2 in~\cite{dover3} shows that there is only one 10-secant. On the other hand, $\S$ could have no lines meeting it in more than 6 points. We deal with these cases separately.

\section{The 10-secant case}
Suppose that a blocking semioval $\S$ of size 25 in $PG(2,11)$ has a 10-secant $\ell_{10}$. Let $Q$ and $R$ be the two points on $\ell_{10}$ that are not in $\S$. If $P$ is a point of $\S$ not on $\ell_{10}$, then the tangent line to $\S$ through $P$ must pass through either $Q$ or $R$; let $\S_Q$ (resp. $\S_R$) be the set of points in $\S$ not on $\ell_{10}$ whose tangent passes through $Q$ (resp. $R$). If we denote the set of points in $\S$ not on $\ell_{10}$ as $\S'$, we have $\S' = \S_Q \dot\cup \S_R$, so we must have $|\S_Q| + |\S_R| = 15$.

Since $\S$ is a blocking set, every line through $Q$ must meet $\S$ in at least one point. One of these is $\ell_{10}$, but the remaining 11 lines through $Q$ must be covered by the 15 points of $\S'$. The lines through $Q$ and points of $\S_Q$ are tangents to $\S$ and thus contain only a single point of $\S_Q$, and no points of $\S_R$; hence there are $|\S_Q|$ of these. A line $m$ through $Q$ containing a point $X$ of $\S_R$ cannot be a tangent to $\S$, since the tangent to $\S$ at $X$ meets $\ell_{10}$ in $R$. Hence $m$ contains at least 2 points of $\S$, but these points cannot be on $\ell_{10}$, and by the above cannot be in $\S_Q$, so $m$ must contain at least 2 points of $\S_R$.

Since every line through $Q$ is covered by $\S$, we have:
$$ 12 \leq 1 + |\S_Q| + \frac12|\S_R|$$

Using the fact that $|\S_Q| = 15 - |\S_R|$, we can substitute in this inequality to obtain $|\S_Q| \geq 7$. However, the exact same argument holds for $R$, showing $|\S_R| \geq 7$, from which we conclude that one of $|\S_Q|$ and $|\S_R|$ is 7 and the other is 8. Since our choice of labelling $Q$ and $R$ is arbitrary, we assume without loss of generality that $|\S_Q| = 8$ and $|\S_R| = 7$.

Since $|\S_Q|$ contains 8 points, eight lines through $Q$ are tangents to $\S$, and one is $\ell_{10}$, so the remaining three lines through $Q$ must meet $\S_R$ in at least 2 points each.  As $|\S_R| = 7$, this implies that these remaining 3 lines consist of two 2-secants and a 3-secant to $\S$. Similar analysis shows that $R$ lies on $\ell_{10}$, seven tangents and four 2-secants to $\S$.

With a structure this well-defined, we have a reasonable hope that a computer search can identify whether or not a blocking semioval of this form exists in $PG(2,11)$. We develop a search strategy by picking items in a potential blocking semioval in order, using Magma~\cite{magma} at each stage to determine the number of projectively inequivalent options we have that respect previous choices.

We begin by coordinatizing $PG(2,11)$ with homogeneous coordinates such that $Q = (1,0,0)$ and $R = (0,1,0)$. The automorphism group leaving $Q$ and $R$ invariant is transitive on lines through $Q$ other than $\ell_{10}$, thus we may assume that the 3-secant to $\S$ through $Q$ is $[0,1,0]$.

\textbf{Pick two other lines through $Q$ to be 2-secants to $\S$.}  We use Magma to compute the group of automorphisms in $PG(2,11)$ that leave $Q$ and $R$ invariant, as well as the line $[0,1,0]$, and calculate the orbits of pairs of lines through $Q$. As a result of this calculation, we find that we may always take one of our 2-secants through $Q$ to be $[0,1,1]$, and then the other must be one of $[0,1,2],[0,1,3],[0,1,5],[0,1,7],[0,1,10]$, giving 5 possible options for this pair of lines.

\textbf{Pick three points on the 3-secant to be in $\S_R$.} For each of the five selections of non-tangent lines through $Q$ above, we again use Magma to calculate the group of automorphisms that leave each of $Q$, $R$, and the three non-tangent lines through $Q$ fixed. This group is doubly-transitive on the points of the 3-secant through $Q$ distinct from $Q$, allowing us to assume that $(0,0,1)$ and $(1,0,1)$ are in $\S$. However, this is the extent to which the automorphism group can help us, and we must consider each of the other 9 points on this line as candidates to be in $\S$, leaving 9 options for this point set.

\textbf{Pick two points on $[0,1,1]$ to be in $\S_R$.} We have already picked three points on the 3-secant to be in $\S_R$, and the lines joining each of these points to $R$ cannot contain any other point of $\S$. This eliminates three points of $[0,1,1]$ from consideration, as well as $Q$, so there are ${8 \choose 2} = 28$ candidate point pairs.

\textbf{Pick two points on the other 2-secant through $Q$ to be in $\S_R$.} As before, we may not pick points of this line that already lie on lines connecting points of $\S_R$ and $R$, which eliminates 6 possible points; there can be no overlap since all of these lines intersect at $R$. Hence there are ${6 \choose 2} = 15$ candidate point pairs.

\textbf{Pick the points of $\S_Q$.} There will be four remaining lines through $R$ which do not pass through a point of $\S_R$. Order them arbitrarily. On the first, pick two points to be in $\S_Q$, noting that the line between either of these points and $Q$ must be a tangent. As above, this choice affects the number of candidate points for each of the remaining 2-secants through $R$. These point selections can be made in  ${8 \choose 2} {6 \choose 2} {4 \choose 2} = 2520$ ways.

There are $47,638,000 \approx 2^{25.5}$ possible configurations, which is a managable number for computer search. For each of these configurations, we must check if the resulting set, when added to the ten points on $\ell_{10}$ distinct from $Q$ and $R$, forms a blocking semioval. However, the way we've constructed the set ensures that it is a blocking set, and that every point of $\S'$ has a unique tangent line; all that remains to check at this stage is whether there is a unique tangent line at each of the points of $\ell_{10}$ distinct from $Q$ and $R$, or equivalently if each of the points of $\ell_{10}$ distinct from $Q$ and $R$ lies on exactly two lines not blocked by $\S'$, as one of those two will be $\ell_{10}$.

We executed this search in Magma, and we found no blocking semioval of this form. Hence we conclude that there is no blocking semioval of size 25 in $PG(2,11)$ with a 10-secant.

\section{Constraint Programming}
\label{sat_solver}
In order to confirm the results of the previous section, and to set up for further searching, we have developed a constraint programming model to search for blocking semiovals, implemented using Google's OR-Tools~\cite{ortools} Boolean satisfiability (SAT) solver, using Python. For each point $(x,y,z)$, we create a Boolean variable $p\_x\_y\_z$ which is true if $(x,y,z)$ is in our blocking semioval and false otherwise. Mechanically, we also need to create a 0/1-valued integer variable tied to the value of the Boolean (0 for false, 1 for true) to calculate line intersection sizes with the blocking semioval, but for purposes of the model description we will conflate these two variables. Similarly, for each line $[a,b,c]$ we create a Boolean variable $\ell\_a\_b\_c$ which is true if and only if the line is tangent to our blocking semioval.

Given these variable definitions, every blocking semioval corresponds to a solution of the constraint program defined with only three types of constraints:
\begin{eqnarray*}
\sum_{(x,y,z) \in [a,b,c]} p\_x\_y\_z &=& 1\; {\rm if} \; \ell\_a\_b\_c\\
\sum_{(x,y,z) \in [a,b,c]} p\_x\_y\_z &>& 1\; {\rm if} \; \neg\ell\_a\_b\_c\\
\sum_{[a,b,c] \ni (x,y,z)} \ell\_a\_b\_c &=& 1\; {\rm if} \; p\_x\_y\_z\\
\end{eqnarray*}
The first equation type asserts that if $\ell$ is a tangent, then only one of its points lies on the blocking semioval. The second equation type asserts that all non-tangent lines meet the blocking semioval in more than one point. The final equation type asserts that for each point of the blocking semioval, only one of the lines through it is a tangent. Applying these conditions across all points and lines in the plane, we see that any solution of this constraint program is a blocking semioval, and vice versa.

The SAT problem is known to be NP-complete in general, but there are numerous algorithms and implementations which can solve many specific instances. In particular, we have modeled the specific search of the last section using constraint programming by adding the additional conditions derived to our generic model: that the total number of points must be 25, that certain points must be or not be on the blocking semioval, and that certain lines must meet the blocking semioval in 1, 2 or 3 points. By way of comparison, the Magma search conducted in the previous section took approximately 53 minutes to determine that there was no blocking semioval of the claimed form. Google's OR-Tools SAT solver reached the same conclusion in 47 seconds; excerpts of the code are given in Appendix A. Using the constraint programming approach will be particularly valuable for us in the next section.

\section{The 6-secant case}
We have concluded that if $\S$ is a blocking semioval $PG(2,11)$ with 25 points, $\S$ cannot have a 10-secant. Hence if $\S$ exists, no line can meet $\S$ in more than 6 points. Let $x_i$ denote the number of lines of $PG(2,11)$ meeting $\S$ in exactly $i$ points. From~\cite{dover2}, we have the following relations, specialized to $PG(2,11)$:
\begin{eqnarray}
x_2 + x_5 + 3x_6 & = & 123 \label{eq1}\\
x_3 - 3x_5 - 8x_6 & = & -89 \label{eq2}\\
x_4 + 3x_5 + 6x_6 & = & 74 \label{eq3}
\end{eqnarray}

Note that Equation~\ref{eq3}, when reduced modulo 3, shows that $x_4 \equiv 2 \pmod{3}$, and in particular must be at least 2. Hence $3x_5 + 6x_6 \le 72$ by Equation~\ref{eq3}. On the other hand Equation~\ref{eq2} shows that $3x_5 + 8 x_6 \ge 89$. Hence $2 x_6 \ge 17$, from which we conclude $x_6 \geq 9$. We proceed through with a series of small propositions that will gradually restrict the possible structure of $\S$.

\begin{prop}\label{prop1}
Let $\S$ be a blocking semioval in $PG(2,11)$ with 25 points and no 10-secant. Then no point of $\S$ lies on more than three 6-secants to $\S$, and at least four points of $\S$ lie on exactly three 6-secants to $\S$.
\end{prop}

\begin{proof}
Let $\L_6$ be the set of 6-secants to $\S$, and define $y_i$ to be the number of points of $\S$ lying on exactly $i$ 6-secants in $\L_6$. First note that $y_i = 0$ for all $i \ge 5$. Indeed if $P \in \S$ lies on five (or more) 6-secants, then each of these 6-secants contains 5 points of $\S$ distinct from $P$, giving at least $25+1 = 26$ points on $\S$.

We also note that $y_4 = 0$. If $P \in \S$ lies on four 6-secants, then these four 6-secants contain 20 points of $\S$ distinct from $P$. One of the other lines through $P$ is a tangent to $\S$ at $P$, leaving seven additional lines through $P$, each of which must contain at least one other point of $\S$ distinct from $P$. This forces $\S$ to have at least 28 points, again a contradiction. Hence no point of $\S$ can lie on more than three 6-secants in $\L_6$.

Counting points of $\S$ and flags of points of $\S$ lying on lines of $\L_6$, we obtain the following relations on the $y_i$:
\begin{eqnarray*}
y_0 + y_1 + y_2 + y_3 & = & 25\\
y_1 + 2y_2 + 3y_3 & = & 6x_6 \ge 54
\end{eqnarray*}

As all of the $y_i$ are non-negative, we certainly have $50 + y_3 = 2y_0 + 2y_1 + 2y_2 + 2y_3 + y_3 \ge  y_1 + 2y_2 + 3y_3 \ge 54$, from which we derive that $y_3 \ge 4$.
\end{proof}

\begin{prop}\label{prop2}
Let $\S$ be a blocking semioval in $PG(2,11)$ with 25 points and no 10-secant. Then there exists a pair of points $Q,R \in \S$ such that $Q$ and $R$ each lie on three 6-secants to $\S$, and the line $\overline{QR}$ is also a 6-secant to $\S$.
\end{prop}

\begin{proof}
Let $\Y_3$ be the set of points in $\S$ that lie on exactly three 6-secants. Let $P_1$, $P_2$, $P_3$, and $P_4$ be four distinct points of $\Y_3$, which must exist by Proposition~\ref{prop1}. If any pair of the $P_i$ lie together on a 6-secant to $\S$, we are done, so assume otherwise. The three 6-secants through $P_1$ each contains 5 points of $\S$ in addition to $P_1$, and we define $\S_1$ to be this set of 15 points lying on a 6-secant through $P_1$, distinct from $P_1$. Note that by assumption none of $P_2$, $P_3$ or $P_4$ can be in $\S_1$, as otherwise two points of $\Y_3$ would lie together on a 6-secant.

Now consider the three 6-secants through $P_2$. Again, none of these lines can contain any of $P_1$, $P_3$ or $P_4$. Moreover, each of these lines meets $\S_1$ in at most three points, one from each of the 6-secants through $P_1$. Hence each 6-secant through $P_2$ contains at least two points distinct from the $P_i$ and $\S_1$. But there are four $P_i$, and 15 points in $\S_1$. Since $\S$ only has 25 points, each 6-secant through $P_2$ must contain exactly two points distinct from the $P_i$ and $\S_1$, and thus exactly three points of $\S_1$. In particular, exactly nine points of $\S_1$ lie on a 6-secant through $P_1$ and a 6-secant through $P_2$. Call these points $\S_{12}$.

The same analysis holds for the three 6-secants through $P_3$. Hence there are exactly nine points of $\S_1$ that lie on a 6-secant through $P_1$ and a 6-secant through $P_3$, which we denote $\S_{13}$. Since $\S_1$ only has 15 points, its 9-point subsets $\S_{12}$ and $\S_{13}$ must intersect in at least three points, and any point of that intersection is a point of $\S$ that lies on exactly three 6-secants (it cannot lie on more, by Proposition~\ref{prop1}) and lies on a 6-secant together with another point of $\Y_3$. This completes the proof.
\end{proof}

%We finish with a quick counting proposition.
%\begin{prop}\label{prop3}
%Let $\S$ be a blocking semioval in $PG(2,11)$ with 25 points and no 10-secant. If $P \in \S$ lies on three 6-secants to $\S$, then $P$ lies on one tangent, seven 2-secants, one 3-secant and %three 6-secants to $\S$.
%\end{prop}

%\begin{proof}
%Let $P$ be as in the proposition statement. Since $\S$ is a blocking semioval, $P$ necessarily lies on a unique tangent line. The remaining 11 lines through $P$ must each contain at least two %points of $\S$. By hypothesis, three of these are 6-secants, leaving eight lines through $P$, each of which contains at least one of the 9 remaining points of $\S$. Hence seven of these %remaining lines must be 2-secants, and one must be a 3-secant.
%\end{proof}

From Proposition~\ref{prop2}, our blocking semioval $\S$ must contain two points $Q$ and $R$ which each lie on three 6-secants, one of which is their common line $n$. Let $\ell_1$ and $\ell_2$ be the other two 6-secants through $Q$ and $m_1$ and $m_2$ be the other two 6-secants through $R$. Define $\I$ to be the set of four points which are the intersections of $\ell_i$ and $m_j$ for $i,j \in \{1,2\}$. Figure~\ref{fig1} provides a graphical depiction of these definitions.

\begin{figure}
\centering
\begin{tikzpicture}
\draw[black,thick] (-6,2) -- (6,2) node[pos=1.05,text=black]{$n$};
\draw[black,thick] (-5.625,2.5) -- (1.25,-3) node[pos=1.05,text=black]{$\ell_1$};
\draw[black,thick] (5.625,2.5) -- (-1.25,-3) node[pos=1.05,text=black]{$m_1$};
\draw[black,thick] (-5.5,2.5) -- (0.5,-3.5) node[pos=1.05,text=black]{$\ell_2$};
\draw[black,thick] (5.5,2.5) -- (-0.5,-3.5) node[pos=1.05,text=black]{$m_2$};
\filldraw[black] (-5,2) circle (4pt) node[anchor=north east]{$Q$};
\filldraw[black] (5,2) circle (4pt) node[anchor=north west]{$R$};
\filldraw[black] (-3,2) circle (4pt);
\filldraw[black] (3,2) circle (4pt);
\filldraw[black] (-1,2) circle (4pt);
\filldraw[black] (1,2) circle (4pt);
\draw[black](0,-2.5) circle [radius=1] node[]{$\I$};
\end{tikzpicture}
\caption{Initial structure of blocking semioval $\S$ based on Proposition~\ref{prop3}}
\label{fig1}
\end{figure}
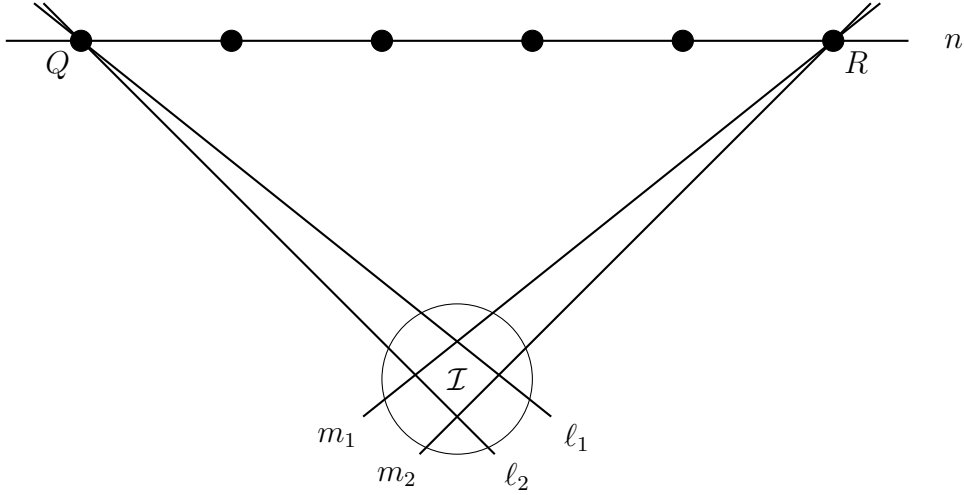

\begin{prop}\label{prop3}
Let $\S$ be a blocking semioval in $PG(2,11)$ with 25 points and no 10-secant. Define the points $Q$ and $R$, the lines $n$, $\ell_1$, $\ell_2$, $m_1$ and $m_2$, and the set $\I$ as above. Then at least three points of $\I$ must lie in $\S$.
\end{prop}

\begin{proof}
We begin by counting the number of points in the subset $\S_6$ of $\S$ consisting of the points of $\S$ that lie on our five noted 6-secants. There are six points on $n$, and then each of $\ell_1$, $\ell_2$, $m_1$ and $m_2$ contributes five additional points (not counting $Q$ and $R$), with the points of $\I$ counted twice. Hence $\S_6$ contains $26 - |\I|$ points, showing immediately that $\I$ contains at least one point of $\S$.

But recall that $\S$ must have at least nine 6-secants. Except for the five 6-secants defining $\S_6$, no 6-secant can contain more than five points of $\S_6$. So all of the additional 6-secants, of which there are at least four, must contain a point of $\S$ that is not in $\S_6$. However, they cannot all contain the same point off $\S_6$, since by Proposition~\ref{prop1} a point of $\S$ can lie on at most three 6-secants. So there must be at least two points of $\S$ not in $\S_6$, forcing $\S_6$ to contain at most 23 points, from which we conclude that at least three points of $\I$ lie in $\S$.
\end{proof}

Define $\S^*$ to be the set of points in $\S$ that lie on one or more of the lines $n$, $\ell_1$, $\ell_2$, $m_1$ and $m_2$. Notice that the configuration given in Figure~\ref{fig1} shows five 6-secants to our putative blocking semioval $\S$ (as well as $\S^*$); but recall that $\S$ must have at least nine 6-secants. As the points of $\S^*$ are on the union of five lines, any 6-secant to $\S$ not shown must contain at least one point not in $\S^*$. On the other hand, $\S^*$ contains either 22 or 23 of the points of the blocking semioval, depending on whether $\I$ has 4 or 3 points in $\S$, respectively. Let's address the latter case first.

\begin{prop}\label{propI=3}
Let $\S$ be a blocking semioval in $PG(2,11)$ with 25 points and no 10-secant. Define the points $Q$ and $R$, the lines $n$, $\ell_1$, $\ell_2$, $m_1$ and $m_2$, and the set $\I$ as above. If $\I$ has three points in $\S$, then there exists a point $P \in \S \setminus \S^*$ which lies on at least two 6-secants to $\S$ such that neither of these 6-secants contains any point of $\S \setminus \S^*$ except $P$.
\end{prop}

\begin{proof}
Since $\I$ has three points in $\S$, $\S^*$ contains 23 points of $\S$, leaving two points of $\S$ not in $\S^*$, which we call $A$ and $B$. Since $\S$ has at least nine 6-secants and only five of these contain points strictly within $\S^*$, there must be at least four 6-secants to $\S$ which contain at least one of $A$ or $B$. It is possible that the line containing $A$ and $B$ is a 6-secant to $\S$, but this leaves three 6-secants to $\S$ that contain exactly one of $A$ or $B$. The pigeonhole principle completes the proof, as one of these two points must lie on two such 6-secants.
\end{proof}

Though a seemingly simple result, Proposition~\ref{propI=3} is incredibly powerful from a computational perspective. In order to set up for a search, let us coordinatize $PG(2,11)$ such that the points of $\I$ are $\ell_2 \cap m_2 = (1,0,0)$, $\ell_2 \cap m_1 = (0,1,0)$, $\ell_1 \cap m_2 = (0,0,1)$ and $\ell_1 \cap m_1 = (1,1,1)$, with $(1,1,1)$ not in $\S$. This forces $Q=(1,1,0)$ and $R=(1,0,1)$ to also be in $\S$. Any automorphism which fixes this configuration must leave $(1,1,1)$ and $(1,0,0)$ fixed, but it could permute $(0,1,0)$ and $(0,0,1)$, thus we do have a non-trivial automorphism group of order 2 which we can use to narrow the search space for our point $P \in \S$ whose existence is guaranteed by Proposition~\ref{propI=3}. Using Magma, we find that there are 45 orbits of candidate points for $P$.

Once $P$ is chosen, we need to pick two lines through it to be 6-secants. But in order to be 6-secants, these lines must meet each of $n$, $\ell_1$, $\ell_2$, $m_1$ and $m_2$ in ${\bf distinct}$ points, which must all be in $\S$. Thus we need to pick two lines through $P$ out of 6 or 7 candidates (depending on whether or not $P$ lies on a line containing two points of $\I$), which can be done in either 15 or 21 ways. Running out all the possibilities with Magma, we find that there are 760 possible configurations for $P$ and its two 6-secants. Combined with the 5 points we coordinatized above, we have 760 starting configurations, each of which provides 16 out of 25 points on a putative blocking semioval. This puts us in the realm of easy computation.

For each of these 760 starting configurations, we create a constraint programming model as in Section~\ref{sat_solver}. Starting with the basic relations used to define a blocking semioval, we add the additional constraints:
\begin{enumerate}
\item the total number of points on the blocking semioval is 25;
\item the points $\{(1,0,0),(0,1,0),(0,0,1),(1,1,0),(1,0,1)\}$ are in the blocking semioval;
\item the point $(1,1,1)$ is not in the blocking semioval;
\item the point $P$ is in the blocking semioval; and
\item the points of intersection of the two lines chosen to be 6-secants to the blocking semioval with each of $\ell_1$, $\ell_2$, $m_1$, $m_2$ and $n$ are in the blocking semioval.
\end{enumerate}

The SAT solver runs through these cases at a rate of roughly two per second, and in each case determined that the model was infeasible, from which we can conclude that there is no blocking semioval of the form discussed here with just three points of $\I$ in the blocking semioval. Now we follow a similar program to  Proposition~\ref{propI=3} for the case where all four points of $\I$ lie in $\S$.

\begin{prop}\label{propI=4}
Let $\S$ be a blocking semioval in $PG(2,11)$ with 25 points and no 10-secant. Define the points $Q$ and $R$, the lines $n$, $\ell_1$, $\ell_2$, $m_1$ and $m_2$, and the set $\I$ as above. If all four points of $\I$ lie in $\S$, then there exists a point $P \in \S \setminus \S^*$ which lies on at least two 6-secants to $\S$ such that at least one of these 6-secants contains no point of $\S \setminus \S^*$ except $P$.
\end{prop}

\begin{proof}
Since all four points of $\I$ lie in $\S$, $\S^*$ contains 22 points of $\S$, leaving three points of $\S$ not in $\S^*$. As before $\S$ has at least nine 6-secants and only five of these contain points strictly within $\S^*$, so there must be at least four 6-secants to $\S$ which contain at least one point not in $\S^*$. Since there are only three points in $\S \setminus \S^*$, at most three of our four 6-secants can contain more than one point of $\S \setminus \S^*$, so there must exist a point $P \in \S \setminus \S^*$ which lies on a 6-secant that contains no other point of $\S \setminus \S^*$.

If there is another 6-secant through $P$ we are done. If no other 6-secant passes through $P$, then the remaining three 6-secants must pass through either of $A$ or $B$, the two points of $\S \setminus \S^*$ distinct from $P$. Again using the pigeonhole principle, at least one of these points (say $A$) must lie on at least two 6-secants, and at most one of those 6-secants can contain another point (necessarily $B$) of $\S \setminus \S^*$. In this case $A$ lies on at least two 6-secants, at least one of which contains no other point of $\S \setminus \S^*$, proving the claim.
\end{proof}

While not as powerful as Proposition~\ref{propI=3}, it gets the job done. We again coordinatize $PG(2,11)$ as above, but in this case all of $(1,0,0)$, $(0,1,0)$, $(0,0,1)$ and $(1,1,1)$ are in $\S$, as well as $(1,1,0)$ and $(1,0,1)$. There is an automorphism group of order 8 that fixes $\{(1,1,0),(1,0,1)\}$ and $\{(1,0,0),(0,1,0),(0,0,1),(1,1,1)\}$ as sets, which allows us to narrow down the choice of our point $P$, whose existence is guaranteed by Proposition~\ref{propI=4}, to one of 15 possible orbits. We then choose a line to be the 6-secant which contains no other point of $\S \setminus \S^*$, which we can pick in 6 or 7 ways, as described above. However, our remaining 6-secant could be any other line through $P$ except $\ell_{10}$, yielding 1,056 cases.

For each of these 1,056 starting configurations, we again create a constraint programming model. In addition to the basic relations used to define a blocking semioval, we add:
\begin{enumerate}
\item the total number of points on the blocking semioval is 25;
\item the points $\{(1,0,0),(0,1,0),(0,0,1),(1,1,1),(1,1,0),(1,0,1)\}$ are in the blocking semioval;
\item the point $P$ is in the blocking semioval;
\item the points of intersection of the first line chosen to be a 6-secant to the blocking semioval with each of $\ell_1$, $\ell_2$, $m_1$, $m_2$ and $n$ are in the blocking semioval; and
\item the second line chosen to be a 6-secant to the blocking semioval contains six points of the blocking semioval.
\end{enumerate}

The SAT solver resolves each of these case in approximately 4 seconds, and again determined that each case was infeasible. This allows us to state our main result:

\begin{thm}
There is no blocking semioval with exactly 25 points in $PG(2,11)$.
\end{thm}

\section{Conclusion}

As part of testing our constraint programming code to double check the 10-secant case, we commented out the constraint forcing the number of points in the blocking semioval to be 25, assuming the code would find a vertexless triangle or some other known blocking semioval. Instead, we found a blocking semioval in $PG(2,11)$ with 26 points, well smaller than the 29 points contained in the smallest previously known blocking semioval in that plane. The coordinates for the points of this blocking semioval are given in Table~\ref{T2}.
\begin{table}[!h]
\centering
\caption{Points of a 26-point semioval in $PG(2,11)$}
\begin{tabular}{ c c c c c c c c c }
$(1,1,0)$ & $(1,2,0)$ & $(1,3,0)$ & $(1,4,0)$ & $(1,5,0)$ & $(1,6,0)$ & $(1,7,0)$ & $(1,8,0)$ & $(1,9,0)$\\
$(1,10,0)$ & $(0,0,1)$ & $(1,0,2)$ & $(1,0,6)$ & $(1,0,7)$ & $(1,0,8)$ & $(1,0,10)$ & $(1,1,3)$ & $(1,2,3)$\\
$(1,3,5)$ & $(1,4,4)$ & $(1,5,9)$ & $(1,6,5)$ & $(1,7,1)$ & $(1,8,4)$ & $(1,9,1)$ & $(1,10,9)$\\
\end{tabular}
\label{T2}
\end{table}

The stabilizer $G$ of this blocking semioval has order 5, and has three fixed points $(1,0,0)$, $(0,1,0)$ and $(0,0,1)$. It is generated by a collineation induced by the matrix:
$$
\left(\begin{matrix}
1 & 0 & 0\\
0 & 9 & 0\\
0 & 0 & 4
\end{matrix}\right)
$$

The three lines of the triangle generated by the fixed points are a 10-secant ($[0,0,1]$), a tangent ($[1,0,0]$, at $(0,0,1)$), and a 6-secant which contains $(0,0,1)$ and a 5-point orbit under $G$. Using this combinatorial structure and an analogous group structure, we attempted to generalize this blocking semioval to larger planes, particularly $PG(2,19)$, using a fairly straightforward exhaustion in Magma using all possible pairs of squares in $GF(19)$ to define our automorphism group. This search was unsuccessful; however, analyzing the list of blocking semiovals in $PG(2,7)$~\cite{ranson} shows that the unique smallest blocking semioval in that plane, of size 16, has an analogous automorphism group and combinatorial structure.

Table~\ref{T1} provides a summary of the size of the smallest blocking semiovals in small Desarguesian planes, where these values are known.

\begin{table}[!h]
\centering
\caption{Sizes of smallest known blocking semiovals}
\begin{tabular}{|c|c||c|c|}
\hline
Plane & Minimum Size & Plane & Minimum Size\\ \hline
$PG(2,2)$ & none & $PG(2,3)$ & 6 \\ \hline
$PG(2,4)$ & 9 & $PG(2,5)$ & 11 \\ \hline
$PG(2,7)$ & 16 & $PG(2,8)$ & 19 \\ \hline
$PG(2,9)$ & 21 & $PG(2,11)$ & 26 \\ \hline
\end{tabular}
\label{T1}
\end{table}

\bibliographystyle{plain}

\section*{Appendix A: Constraint programming primitives}
The constraint program needed to define the incidences in $PG(2,11)$ is rather lengthy; we actually developed a Magma program to write the constraint program. Here we provide some excerpts from the constraint programs used in this paper. The code here is excerpted from a Python program using Google's OR Tools library. Note that sums with ellipses continue over all points on a line or lines through a point, and are redacted for readability considerations.
\begin{Verbatim}
###
#Variables describing a point
###
# Boolean is true/false as (1,0,0) is in/not in the blocking semioval
p_1_0_0 = model.NewBoolVar("point_1_0_0")
# Integer value of the Boolean variable
p_1_0_0_int = model.NewIntVar(0,1,"P![1,0,0]")
#Constraints to tie the Boolean and integer variables together
model.Add(p_1_0_0_int == 1).OnlyEnforceIf(p_1_0_0)
model.Add(p_1_0_0_int == 0).OnlyEnforceIf(p_1_0_0.Not())

###
#Variables describing a line
###
# Boolean is true/false as [1,0,0] is/is not tangent to blocking semioval
l_1_0_0 = model.NewBoolVar("line_1_0_0")
# Integer value of the Boolean variable
l_1_0_0_int = model.NewIntVar(0,1,"line_1_0_0_int")
#Constraints to tie the Boolean and integer variables together
model.Add(l_1_0_0_int == 1).OnlyEnforceIf(l_1_0_0)
model.Add(l_1_0_0_int == 0).OnlyEnforceIf(l_1_0_0.Not())

###
#Blocking semioval constraints
###
# Only one point on [1,0,0] is on the blocking semioval if it is a tangent
model.Add(p_0_1_0_int+p_0_1_1_int+ ... == 1).OnlyEnforceIf(l_1_0_0)
# More than one point on [1,0,0] is on the blocking semioval if not
model.Add(p_0_1_0_int+ ... > 1).OnlyEnforceIf(l_1_0_0.Not())
# Only one line through (1,0,0) is tangent to the blocking semioval
model.Add(l_0_1_0_int+l_0_1_1_int+ ... == 1).OnlyEnforceIf(p_1_0_0)

###
#Additional constraints (not necessarily used simultaneously)
###
# Force [0,0,1] to be a 6-secant
model.Add(p_1_0_0_int+p_1_1_0_int+p_0_1_0_int+ ... == 6)
# Assert (1,0,0) is in the blocking semioval, while (1,1,1) is not
model.AddBoolAnd([p_1_0_0,p_1_1_1.Not()])
# Assert the blocking semioval has 25 points
model.Add(p_1_0_0_int+p_0_1_0_int+p_0_0_1_int+ ... == 25)

\end{Verbatim}

\end{document}